\documentclass[fontsize=12pt]{amsart}

\usepackage{amsmath}
\usepackage{amsthm}
\usepackage{amssymb}
\usepackage{mathtools}
\usepackage{mathrsfs}
\usepackage{physics}
\usepackage[margin=1in]{geometry}
\usepackage{scrextend}
\usepackage{tikz-cd}
\usepackage{tikz}
\usepackage[shortlabels]{enumitem}
\usepackage{blkarray}
\usepackage{eso-pic}
\usepackage[outline]{contour}
\usepackage{comment}

\edef\restoreparindent{\parindent=\the\parindent\relax}
\usepackage{parskip}
\restoreparindent

\usepackage[hyperfootnotes=false, pdfencoding=auto]{hyperref}
\hypersetup{
  colorlinks= true,
  linkcolor=[RGB]{0,89,42},
  urlcolor=[RGB]{0,89,42},
  citecolor=[RGB]{0,89,42}
}

\usetikzlibrary{tikzmark,fit,decorations.pathmorphing}

\usepackage[notcolon, notquote]{hanging}

\newtheorem{Lthm}{Theorem}

\newtheorem{theorem}{Theorem}
\numberwithin{theorem}{section}

\newtheorem{lemma}[theorem]{Lemma}
\newtheorem{proposition}[theorem]{Proposition}

\theoremstyle{definition}
\newtheorem{remark}[theorem]{Remark}
\newtheorem{question}[theorem]{Question}
\newtheorem{problem}[theorem]{Problem}
\newtheorem{example}[theorem]{Example}
\newtheorem{definition}[theorem]{Definition}

\renewcommand{\Tr}{\text{Tr}}

\renewcommand{\P}{\mathbb{P}}

\newcommand{\Sym}{\textup{Sym}}
\newcommand{\Alb}{\textup{Alb}}

\usepackage{pdflscape}

\renewcommand{\d}{\mathfrak{d}}

\DeclareMathOperator{\gon}{gon}
\DeclareMathOperator{\irr}{irr}
\DeclareMathOperator{\pr}{pr}
\DeclareMathOperator{\Supp}{Supp}
\DeclareMathOperator{\id}{id}

\newcommand{\mybigwedge}{\raisebox{.32ex}{\scalebox{0.75}{$\bigwedge$}}}

\newlength{\LETTERheight}
\AtBeginDocument{\settoheight{\LETTERheight}{I}}
\newcommand*{\longleadsto}[1]{\ \raisebox{0.24\LETTERheight}{\tikz \draw [->,
line join=round,
decorate, decoration={
    zigzag,
    segment length=6,
    amplitude=2,
    post=lineto,
    post length=2pt
}] (0,0) -- (#1,0);}\ }

\numberwithin{equation}{section}

\begin{document}

\title{Rational maps from products of curves to surfaces with $p_g = q = 0$}

\author{Nathan Chen and Olivier Martin}


\address{Department of Mathematics, Harvard University, Cambridge, MA 02138}
\email{nathanchen@math.harvard.edu}

\address{Department of Mathematics, Stony Brook University, Stony Brook, NY 11794-3651}
\email{olivier.martin@stonybrook.edu}

\thanks{The first author's research was partially supported by an NSF postdoctoral fellowship, DMS-2103099.}

\maketitle

\begin{abstract}
We study dominant rational maps from a product of two curves to surfaces with $p_{g} = q = 0$. Given two curves which satisfy a mild genericity assumption and have large genus relative to their gonality, we show that the degree of irrationality of their product is equal to the product of their gonalities. Moreover, we prove that the degree of irrationality of a product of two hyperelliptic curves is 4.
\end{abstract}

\section*{Introduction}

Dominant rational maps from a product of two curves $C_1$ and $C_2$ to a projective surface $S$ have been studied by Bastianelli, Lee, and Pirola in \cite{BP13,LP16}. Lee and Pirola prove that if $g(C_1)\geq 7$, $g(C_2)\geq 3$, and $C_1, C_2$ are very general then $S$ must be ruled. Instead of describing restrictions on the geometry of $S$, we assume $p_{g}(S)=q(S) = 0$ and give sharp lower bounds on the degrees of such maps when the gonal maps of $C_1$ and $C_2$ do not factor and these curves have large genus relative to their gonality. Recall that a \textit{gonal map} for a curve $C$ is a map $C \longrightarrow \P^{1}$ of minimal degree.

Our motivation stems in part from the desire to understand the \textit{degree of irrationality} of a product of two curves. Recall that the degree of irrationality of a projective variety $X$ is defined as:
\[ \irr(X)  \ =_{\text{def}}  \ \min \left\{ \delta > 0 \mid \exists \text{ dominant rational map } X \dashrightarrow \P^{n}\text{ of degree } \delta \right \}. \]
This birational invariant generalizes the classical notion of \textit{gonality} of an algebraic curve to varieties of higher dimension and is equal to $1$ exactly when $X$ is rational. A considerable amount of recent work has been devoted to estimating and calculating this invariant for various classes of algebraic varieties -- see \cite{BDELU17} and \cite[\S2.2]{MThesis} for a survey. Most relevant to the subject at hand are bounds for measures of irrationality of the symmetric square of a curve given by Bastianelli \cite{Bastianelli12}. For products of curves of genus $\geq 2$, there are naive bounds
\[ \gon(C_{1}) \cdot \gon(C_{2})  \ \geq \  \irr(C_{1} \times C_{2}) \  \geq \  \max\{ \gon(C_{1}), \gon(C_{2}) \} \]
coming from the product of the gonality maps and \cite[Theorem 2]{HM82}.

We first prove that the this upper bound is achieved for products of hyperelliptic curves.\footnote{By hyperelliptic curve, we mean a curve of genus at least $2$ with a $g^{1}_{2}$.}

\begin{Lthm}\label{thm:hyp}
Consider a dominant rational map $\varphi \colon C_{1} \times C_{2} \dashrightarrow S$ of degree $3$ from a product of curves to a surface with $p_{g}(S) = q(S) = 0$ and suppose that $g(C_1)\leq g(C_2)$. Then $g(C_{1}) \leq 1$. In particular, the degree of irrationality of a product of two hyperelliptic curves is $4$.
\end{Lthm}

\noindent Theorem~\ref{thm:hyp} should be contrasted with the situation for products of elliptic curves. Yoshihara has provided examples of products of elliptic curves which admit dominant rational maps of degree 3 to $\mathbb{P}^2$. While it is believed that the degree of irrationality of the product of two very general elliptic curves is $4$, this question remains open.

Our main result is that this upper bound for $\irr(C_{1} \times C_{2})$ is achieved for most pairs of curves whose genera are sufficiently large relative to their gonalities.

\begin{Lthm}\label{thm:largeg}
Let $C_1$ and $C_2$ be curves of gonality $a_1$ and $a_2$ with gonal maps $f_1$ and $f_2$ which do not factor, and let $\varphi: C_1\times C_2\dashrightarrow S$ be a dominant rational map of degree $k$ to a surface $S$ with $p_{g}(S) = 0$. Suppose that $g(C_i) > (a_i-1)(\max\{k, a_{i}\}-1)$ for $i = 1,2$. Then
\begin{enumerate}[label=\textup{(\roman*)}]
    \item $\varphi$ factors through one of the maps:
\[ f_1\times \textup{id}_{C_2} \colon C_1\times C_2 \longrightarrow \mathbb{P}^1\times C_2, \qquad \textup{id}_{C_1} \times f_{2} \colon C_1\times C_2 \longrightarrow C_1\times \mathbb{P}^1. \]
\item Moreover, if $q(S)=0$ then $k \geq a_1 a_2$. In particular, $\irr(C_{1} \times C_{2}) = a_{1}a_{2}$.
\end{enumerate}
\end{Lthm}

\noindent The technical assumption that the gonal maps $f_{1}$ and $f_{2}$ do not factor is necessary (see Example~\ref{example:EllipticCover}).

\begin{remark}\label{factoremark}
The gonal map of a very general $a$-gonal curve of genus $g>(a-1)^2$ does not factor. Indeed, consider the Hurwitz scheme
\[ \mathcal{H}_{g,a} =\{ [C\longrightarrow \P^1] : g(C)=g, \ \deg \varphi=a \text{ and }\varphi\text{ is simply ramified}\}. \]
A simply ramified map $C \longrightarrow \P^1$ of degree $a$ cannot factor nontrivially, and if $g(C)>(a-1)^2$ then the Castelnuovo-Severi inequality (Lemma~\ref{lem:maps2curves}) implies that there is a unique such map $C \longrightarrow \P^{1}$ of degree $\leq a$. This allows us to define a generically injective morphism $\mathcal{H}_{g,a} \longrightarrow \mathcal{M}_{g}^{\gon \leq a}$, where 
\[ \mathcal{M}_{g}^{\gon \leq a} = \{[C]\in \mathcal{M}_g: \text{gon}(C)\leq a\} \subseteq \mathcal{M}_g \]
and it suffices to observe that $\mathcal{M}_{g}^{\gon \leq a}$ is irreducible \cite{AC81} and of the same dimension as $\mathcal{H}_{g,a}$.
\end{remark}

\subsection*{Outline} In $\S$1 we collect some examples and elementary results regarding dominant rational maps to varieties without 1-forms or 2-forms and present several tools which will be used in the proofs of Theorem \ref{thm:hyp} and Theorem \ref{thm:largeg}. The proofs of these theorems will occupy $\S$2. Finally, in $\S$3 we propose some questions and open problems. Throughout the paper, we work over $\mathbb{C}$.

\subsection*{Acknowledgements}

We would like to thank Rob Lazarsfeld for valuable discussions and for suggesting that one could trace $2$-forms to study low degree rational maps.

\section{Background}

\subsection{Miscellaneous results about rational maps from products of curves}In this subsection we gather some results about rational maps from products of curves which complement Theorems~\ref{thm:hyp} and \ref{thm:largeg}. First, given a product of curves $C_{1} \times \cdots \times C_{n}$, it is well-known (see for instance \cite[Theorem 2]{HM82}) that
\begin{equation}\label{maxbound} \irr(C_{1} \times \cdots \times C_{n}) \ \geq \ \max \{ \gon(C_{1}), \ldots, \gon(C_{n}) \}. \end{equation}

Moreover, products of curves do not have interesting dominant rational maps of degree $2$ to varieties without $1$-forms or $2$-forms.

\begin{lemma}\label{deg2maps}
Consider curves $C_1,\ldots, C_n$ of genera $g(C_1)\leq\cdots\leq g(C_n)$ and suppose that
\[ \varphi: C_1\times \cdots\times C_n \dashrightarrow Y \]
is a dominant rational map of degree 2 to a variety with $h^{2,0}(Y)=h^{1,0}(Y)=0$. Then $g(C_{1})=\cdots=g(C_{n-1})=0$ and $C_n$ is rational or hyperelliptic.
  
\end{lemma}

\begin{proof}
The map $\varphi$ gives a birational involution
\[ \iota : C_1\times\cdots\times C_n \dashrightarrow  C_1\times\cdots\times C_n \]
and thus an involution $\iota^*$ of $V =_{\text{def}} H^1(C_1\times\cdots\times C_n,\mathbb{C})$. Since $H^1(Y)=0$, the 1-eigenspace for $\iota^*$ must be trivial and thus the $(-1)$ generalized eigenspace for $\iota^*$ is all of $V$. Accordingly, $\iota^*$ acts on $\mybigwedge^2 V$ with eigenvalue $1$. By the K\"unneth formula, we have the following surjective map which commutes with $\iota^*$:
\[
\mybigwedge^2 V \longrightarrow H^2(C_1\times \cdots \times C_n,\mathbb{C}).
\]
This ensures that $1$ is the only eigenvalue of $\iota^*$ on $H^2(C_1\times \cdots \times C_n,\mathbb{C})$. 
Since $H^{2,0}(Y) = 0$, the space $H^2(C_1\times \cdots \times C_n,\mathbb{C})$ must vanish. It follows that $g(C_1)=\cdots=g(C_{n-1})=0$ and the result is then a consequence of \eqref{maxbound}.

\end{proof}

The following example provides some instances in which the inequality
\[ \irr(C_1\times C_2) \ \leq \  \gon(C_1)\cdot \gon(C_2) \]
is strict, and highlights the importance of the non-factoring assumptions on the gonal maps of $C_1$ and $C_2$ in Theorem \ref{thm:largeg}.
\begin{example}[Product of curves with low degree of irrationality]\label{example:EllipticCover}
Yoshihara \cite{Yoshihara96} has provided examples of products of elliptic curves $E_1\times E_2$ with degree of irrationality $3$. Hence, a product of curves can have degree of irrationality strictly less than the product of the gonality of the factors. For examples involving curves of higher genus, one can use this example in conjunction with the results of Kato-Martens \cite{KM14} and Keem-Martens \cite{KM15}. They give covers $C_{i}
\longrightarrow E_{i}$ of degree $d_{i}$, genus $g(C_{i}) \gg d_{i}$, and gonality $2d_i$ for $i = 1, 2$. Then
\[ \irr(C_1\times C_2) \ \leq \ 3d_1d_2 \ < \ 4d_1d_2 \ = \ \gon(C_1) \cdot \gon(C_2). \]
\end{example}

\subsection{Tools used in the proofs of Theorems~\ref{thm:hyp} and \ref{thm:largeg}}

A key tool we will use to prove part (1) of Theorem \ref{thm:largeg} is the following lemma which states that an appropriate partition of the fibers of a rational map gives rise to a factorization.

\begin{lemma}\label{facto}
Consider generically finite dominant rational maps $f \colon X \dashrightarrow Y$ and $g \colon X \dashrightarrow W$. Suppose that for a general element $y \in Y$, the fiber $f^{-1}(y)$ decomposes into a disjoint union
\[ f^{-1}(y) \ = \ \coprod_{i} g^{-1}(w_{i}) \]
for some finite collection of points $w_{i} \in W$. Then there exists a generically finite rational map $h \colon W \dashrightarrow Y$ making the following diagram commute:
\[\begin{tikzcd}
X \ar[r,dashed,"f"]\ar[d,swap,dashed,"g"]& Y\\
W  \ar[ur,swap,dashed,"h"].&
\end{tikzcd}\]
\end{lemma}

\begin{proof}
After replacing $X, Y, W$ with open subsets, we may assume that both $f$ and $g$ are \'etale morphisms and that the hypothesis holds for all $y \in Y$. Consider the incidence variety
\[ I \ =_{\textup{def}} \ \{ (w, y) \in W \times Y \colon g^{-1}(w) \subseteq f^{-1}(y) \}. \]
Note that this is an algebraic subvariety of $W \times Y$. The projection of $I$ to $W$ is an isomorphism and $I$ dominates $Y$. Therefore we may define $h$ as the composition of this isomorphism with the second projection 
\[h: W\cong I\xrightarrow{ \ \textup{pr}_2 \ } Y\]
and one easily verifies that $f=h\circ g$.
\end{proof}

The following lemma will be used frequently in order to rule out the existence of a $g^1_{k}$ which does not factor through the gonal map of a curve $C$, when $k$ is small relative to $g(C)$.

\begin{lemma}[Castelnuovo-Severi inequality {\cite[VIII, Ex. C-1]{ACGH85}}]\label{lem:maps2curves}
Let $X, C_{1}, C_{2}$ be curves of respective genera $g,g_{1}, g_{2}$. Assume that $f_{i}: X \longrightarrow C_{i}$ is a degree $a_{i}$ cover (for $i = 1, 2$) such that the morphism $f_{1} \times f_{2} \colon X \longrightarrow C_{1} \times C_{2}$ is birational onto its image. Then
\[ g \leq a_{1} g_{1} + a_{2} g_{2} + (a_{1}-1)(a_{2}-1). \]
\end{lemma}

This can be used to prove the nonexistence of certain dominant rational maps whose degrees are coprime to the gonality of one of the curves:

\begin{proposition}
Let $C_{1}, C_{2}$ be curves of positive genus such that $\gon(C_{1}) = d$ and $g(C_{1}) > (d-1)(e-1)$, and let $e \geq 2$ be a positive integer which is relatively prime to $d$. Then there are no dominant rational maps of degree $e$ from $C_{1} \times C_{2}$ to a surface $S$ with $q(S) = 0$.
\end{proposition}

\begin{proof}
Consider a dominant rational map $\varphi: C_1\times C_2\dashrightarrow S$ of degree $e$, where $S$ is a smooth projective surface with $q(S)=0$. The rational map $\varphi$ gives rise to a rational map
\[ \def\arraystretch{1.1}
\begin{array}{c l c l c}
    S & \dashrightarrow & \Sym^{e}(C_{1} \times C_{2}) & \longrightarrow & \Sym^{e}(C_{1}) \\
    t & \longmapsto & \varphi^{-1}(t) & \longmapsto & \pr_{C_{1}}(\varphi^{-1}(t)) 
\end{array}
\]
whose image is contained in a basepoint-free $g^r_e$ for some positive integer $r$. A general line in this basepoint-free $g^r_e$ gives a basepoint-free $g^1_e$ on $C_1$. Lemma \ref{lem:maps2curves} then implies that this $g^1_e$ and a gonal map for $C_1$ factor non-trivially through the same curve, so $\gcd(d,e)\neq 1$.
\end{proof}

\begin{remark}
Given a projective variety $X$ of dimension $n$, one can consider the threshold
\[ \delta_0(X) \ =_{\textup{def}} \ \min \{ \delta_{0} > 0 \mid \forall \delta \geq \delta_0, \ \exists \text{ dominant rational map } X \dashrightarrow \mathbb{P}^n \text{ of degree } \delta \}. \]
It is straightforward to show that $\delta_{0}(X)$ exists. The previous proposition illustrates that there is no upper bound on $\delta_0(C_1\times C_2)$ which depends only on the gonality of $C_1$ and the gonality of $C_2$.
\end{remark}

\begin{lemma}\label{lem:pq}
Let $m$ be a positive integer and $C$ be a curve of gonality $a > 2$ and genus \[ g(C) > (ma-1)(a-1). \]
Assume that $C$ has a gonal map which does not factor. Given elements $D_1,\ldots, D_m$ of the (unique) $g^1_{a}$ on $C$ and $p,q\in \Supp(D_m)$, there exists a 1-form on $C$ vanishing at all of the points of $\bigcup_{i=1}^m\Supp(D_i)$ except for $p$ and $q$.
\end{lemma}

\begin{proof}
By our hypothesis, the gonal map $C \longrightarrow \P^{1}$ of degree $a$ does not factor so Lemma~\ref{lem:maps2curves} together with the bound $g(C) > (ma-1)(a-1)$ implies that (1) there is a unique line bundle $L$ of degree $a$ such that $\abs{L}$ is a base-point free $g^{1}_{a}$, and (2) any base-point free pencil of degree $\leq ma$ must factor through the gonal map. Since $\abs{mL}$ is base-point free, given any element $D \in \abs{mL}$ we can find a base-point free pencil $\mathfrak{d}$ passing through $D$. By (2), $\mathfrak{d}$ factors as
\begin{center}
\begin{tikzcd}
C \arrow[rr, bend right, swap, "\d"] \arrow[r, "\abs{L}"] & \P^{1} \arrow[r, "\alpha"] & \P^{1}
\end{tikzcd}
\end{center}
where $\alpha$ has degree $a$. Since $D$ was arbitrary, this implies that the addition map
\[ \abs{L}^{m} \longrightarrow \abs{mL} \]
is surjective. Because this map is finite onto its image $\dim \abs{mL} = m$ and $\dim \abs{mL} - \dim \abs{(m-1)L} = 1$.

With $D_{1}, \ldots, D_{m}$ in the $g^{1}_{a}$ and $p, q \in \Supp(D_{m})$ as above, we need to show that
\[ h^{0}(K_{C} - (mL - p - q)) - h^{0}(K_{C}-mL) \geq 1. \]
By Serre duality and Riemann-Roch, it suffices to show that $h^{0}(mL) - h^{0}(mL - p - q) \leq 1$. This immediately follows from the fact that $\abs{mL} - \abs{(m-1)L} = 1$ and
\[ (m-1)L \leq mL - p - q. \qedhere \]
\end{proof}

\section{Proofs of Theorems \ref{thm:hyp} and \ref{thm:largeg}}

\subsection{Proof of Theorem~\ref{thm:hyp}}

By Lemma~\ref{deg2maps}, we know that any dominant rational map
\[ \varphi \colon C_{1} \times C_{2} \dashrightarrow S \]
from a product of curves with $g(C_{1}), g(C_{2}) \geq 1$ to a surface $S$ satisfying $p_{g}(S) = q(S) = 0$ must have degree at least $3$. Suppose for contradiction that $g(C_{1}), g(C_{2}) \geq 2$ and that there exists such a map $\varphi$ with $\deg(\varphi) = 3$. The induced map
\[ F_{\varphi} \colon S \dashrightarrow \Sym^{3}(C_{1} \times C_{2}) \longrightarrow \Sym^{3}(C_{i}) \longrightarrow \Alb(C_{i}) \]
factor through $\Alb(S) = 0$. Hence, the image of $S$ in $\Sym^{3}(C_{i})$ is a surface or a rational curve in a fiber of the Albanese map, which is a linear system $\abs{L_{i}}$ for some line bundle $L_{i}$ of degree 3.

Fix a general $s \in S$ and consider $\varphi^{-1}(s) = \{ p_{1}, p_{2}, p_{3} \}$. By Lemma \ref{lem:deg}, \[ \pr_{1}(\varphi^{-1}(s)) \subseteq C_{1} \qquad \text{and} \qquad \pr_{2}(\varphi^{-1}(s)) \subseteq C_{2} \]
both consist of three distinct points. For each $i$, $\pr_{i}(\varphi^{-1}(s))$ cannot contain a pair of points which are images of one another under the hyperelliptic involution. Indeed, the hyperelliptic pencil is unique and such a pair would single out the third point, thereby violating the irreducibility of $C_1\times C_2$.  Assuming $g(C_{1}), g(C_{2}) \geq 2$, we may therefore choose a 1-form $\omega_{1}$ on $C_{1}$ which vanishes at $\pr_{1}(p_{1})$ but is nonzero at $\pr_{1}(p_{2})$ and $\pr_{1}(p_{3})$. Similarly, we may choose a 1-form $\omega_{2}$ on $C_{2}$ which vanishes at $\pr_{2}(p_{2})$ but is nonzero at $\pr_{2}(p_{1})$ and $\pr_{2}(p_{3})$. The 2-form $\pr_{1}^{\ast}\omega_{1} \wedge \pr_{2}^{\ast}\omega_{2}$ then vanishes at $p_{1}$ and $p_{2}$ but not $p_{3}$. By applying Mumford's trace map \cite{Mumford69} (see \cite{BDELU17} for another treatment), we see that
\[ 0 \not= \Tr_{\varphi}(\pr_{1}^{\ast}\omega_{1} \wedge \pr_{2}^{\ast}\omega_{2}) \in H^{0}(S, \Omega^{2}), \]
which contradicts $p_{g}(S) = 0$.

\subsection{Proof of Theorem~\ref{thm:largeg}}

Let us first show how $(ii)$ follows from $(i)$. Suppose that $p_{g}(S)=q(S)=0$. It suffices to show that any dominant rational map $\psi: C_i\times \mathbb{P}^1\dashrightarrow S$ has degree at least $a_i$ for $i=1,2$. Given such a $\psi$, consider the rational map
\begin{align*}F_\psi:  \ S \ &\dashrightarrow  \ \Sym^{\deg\psi}(C_i\times \mathbb{P}^1)\\
t \ &\longmapsto\ \ \ \ \ \ \  \psi^{-1}(t).\end{align*}
Since the composition
\[ S\dashrightarrow \Sym^{\deg\psi}(C_i\times \mathbb{P}^1)\longrightarrow \Sym^{\deg\psi}(C_i)\longrightarrow \text{Alb}(C_i) \]
factors through $\text{Alb}(S)=0$, the positive-dimensional image of $S$
in $\Sym^{\deg\psi}(C_i)$ is contracted by the Albanese map. It follows that $\deg \psi\geq a_i$.

It remains to prove $(i)$. The assumption that the gonal maps do not factor implies by Lemma~\ref{lem:maps2curves} that each curve $C_{i}$ has a unique gonal map of degree $a_{i}$ for $i = 1, 2$. Let $\abs{L_i}=|f_i^*\mathcal{O}_{\mathbb{P}^1}(1)|$ be the corresponding linear system on $C_i$. Recall that $\varphi \colon C_{1} \times C_{2} \dashrightarrow S$ is a dominant rational map.

\begin{lemma}\label{lem:deg}
For a general $s\in S$, the restrictions of the projection maps $\textup{pr}_i: C_1\times C_2\longrightarrow C_i$ to $\varphi^{-1}(s)$ are coverings of their images. In other words, the number
\[ d_{i,x} =_{\textup{def}} \#\{z\in \varphi^{-1}(s): \textup{pr}_i(z)=\textup{pr}_i(x)\} \]
is independent of $x\in \varphi^{-1}(s)$ for a general $s \in S$.
\end{lemma}
\begin{proof}
Let $U\subseteq C_1\times C_2$ be an open subset on which $\varphi$ restricts to an \'etale morphism. Fix $i \in \{ 1, 2 \}$. For each $r\in \mathbb{Z}_{>0}$, consider the quasi-projective subvariety
\[ U_{r}^{i}=_{\textup{def}}\{x\in U: d_{i,x}=r\}\subseteq U. \]
Since $U$ is irreducible and $\bigcup_{r > 0} U_{r}^{i}=U$, there is an integer $r_0$ such that $U_{r_0}^{i}$ is open in $U$. If $Z\subseteq U$ is the complement of $U_{r_0}$ then $\varphi^{-1}(\varphi(Z))\subseteq U$ is a proper closed subset of $U$. Given any $s\in U_{r_0}\setminus \varphi^{-1}(\varphi(Z))$ and any $x\in \varphi^{-1}(x)$, we have $d_{i,x}=r_0$.
\end{proof}

Throughout, we will let $d_{1}$ be the value of $d_{1, x}$ for $x \in \varphi^{1}(s)$ and for a general $s \in S$. There are several cases to consider:

\begin{lemma}\label{lem:cases}
Under the assumptions of Theorem \ref{thm:largeg}, for a general $s \in S$ either:
\begin{enumerate}[label=\textup{(\arabic*)}]
    \item $\textup{pr}_1(\varphi^{-1}(s))$ spans a linear subspace of dimension $k/d_1-1=\# \textup{pr}_1(\varphi^{-1}(s))-1$ in canonical space, or
     \item $\textup{pr}_1(\varphi^{-1}(s))$ is uniquely partitioned into elements of $|L_1|$.
\end{enumerate}
Moreover, case \textup{(2)} breaks up into two subcases. For each element $D$ of $|L_1|$ which is contained in $\textup{pr}_1(\varphi^{-1}(s))$, the image of $\textup{pr}_1|_{\varphi^{-1}(s)}^{-1}(D)$ under the projection to $C_2$ either
\begin{enumerate}
    \item[\textup{(2a)}] does not contain an element of $|L_2|$, or
    \item[\textup{(2b)}] is a union of elements of $|L_2|$.
\end{enumerate}
\end{lemma}

\begin{figure}[htbp!]
\includegraphics[width=12cm]{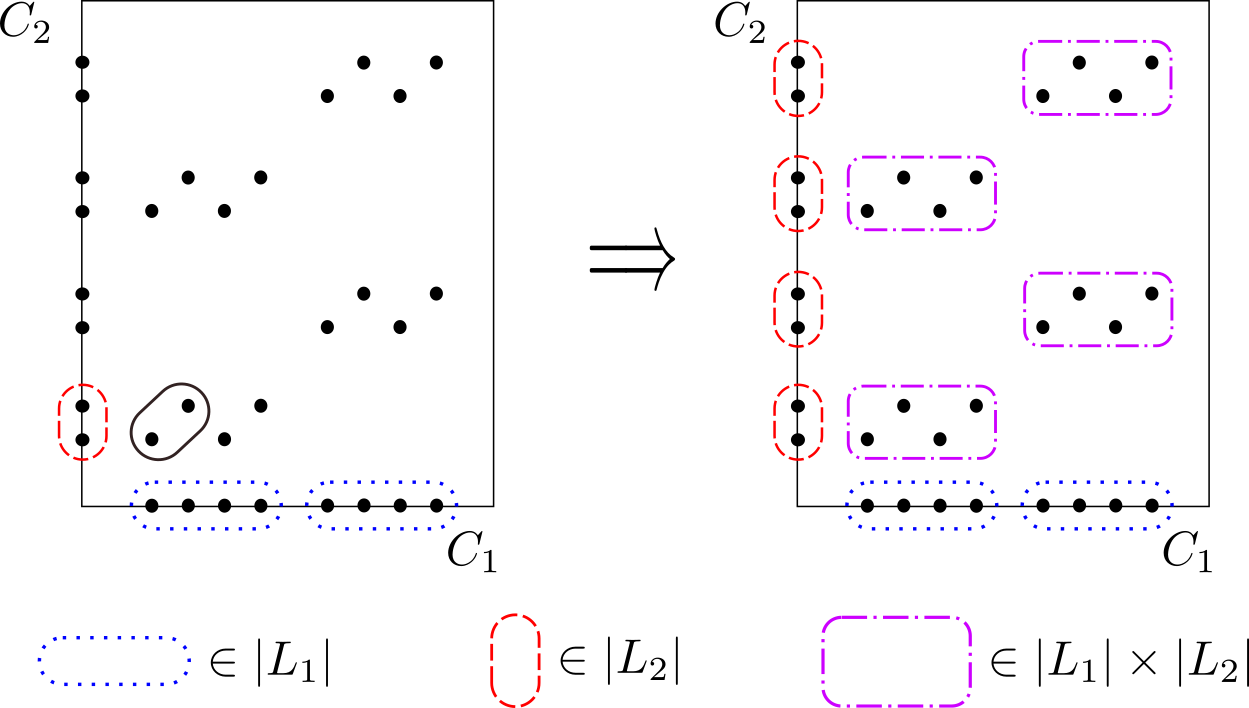}
\caption{The diagram above depicts case (2b) of Lemma~\ref{lem:cases}.}
\end{figure}

Let us introduce some additional notation which will help streamline our arguments. Consider the open subset $U\subseteq C_1\times C_2$ such that $\varphi|_{U}$ is an \'etale morphism and write
\[ F_x=_{\textup{def}}\varphi^{-1}(\varphi(x)) \qquad \text{for } x \in U. \]
By shrinking $U$ if needed, we may assume that $\text{pr}_1(F_x)$ is uniquely partitioned into elements of $|L_1|$ for all $x\in U$. Given $x\in U$, let $D_{i,x}$ be the element of $|L_i|$ satisfying
\[ \pr_{i}(x)\subseteq D_{i,x}\subseteq F_x. \]

\begin{proof}[Proof of Lemma~\ref{lem:cases}]
By an irreducibility argument similar to the one used in Lemma \ref{lem:deg}, if $\text{pr}_1(\varphi^{-1}(s))$ contains an element of $|L_{1}|$ then every point of $\text{pr}_1(\varphi^{-1}(s))$ is contained in an element of $|L_1|$. Since elements of $|L_1|$ are fibers of the gonal map, they have empty intersection. Hence $\text{pr}_1(\varphi^{-1}(s))$ is uniquely partitioned into elements of $|L_1|$.

Now suppose that $\text{pr}_1(\varphi^{-1}(s))$ does not contain an element of $|L_1|$. We claim that these points span a linear subspace of dimension 
\[ k/d_{1} = \# \text{pr}_1(\varphi^{-1}(s))-1\]
in canonical space. Indeed if these points spanned a smaller linear subspace, then by geometric Riemann-Roch they belong to a $g^r_{k/d_{1}}$ on $C_1$ for some $r\geq 1$. Any divisor in this $g^r_{k/d_{1}}$ belongs to a $g^1_{k/d_{1}}$. If the genus of $C_1$ is larger than $(a_{1} - 1) (k/d_{1} - 1)$, then any element of a $g^1_{k/d_{1}}$ must contain an element of $|L_1|$, which provides the desired contradiction.

As for the second claim, suppose that we are in case (2). Consider the closed subvariety
\[ V=_{\textup{def}}\left\{x\in U: \text{pr}_1|_{F_x}^{-1}(D_{1,x}) \text{ contains an element of }|L_2|\right\}\subseteq U. \]
Either $V=U$, in which case we are in situation (2b), or $V$ is a proper closed subset of $U$, in which case we are in situation (2a).
\end{proof}

The proof of Theorem \ref{thm:largeg} is easier in cases (1) and (2a), so we first single out these cases:

\begin{proposition}
Theorem \ref{thm:largeg} holds if $\varphi: C_1\times C_2\dashrightarrow S$ falls under case \textup{(1)} or case \textup{(2a)} of Lemma \ref{lem:cases}.
\end{proposition}

\begin{proof}
Let us first take care of case (1). Given a generic $s\in S$ and any point $x\in \textup{pr}_1(\varphi^{-1}(s))$, if $g(C_1)>k/d_{1}-1$ then we can find a $1$-form $\omega_1$ on $C_1$ which vanishes at all the points in $\text{pr}_1(\varphi^{-1}(s))$ aside from $x$. Indeed, the points of $\textup{pr}_1(\varphi^{-1}(s))$ are linearly independent so $\textup{pr}_1(\varphi^{-1}(s))\setminus \{x\}$ spans a subspace of dimension $k/d_{1}-2$ in canonical space. Now a generic hyperplane in canonical space containing this subspace will not contain $\pr_1(\varphi^{-1}(s))$.

By an irreducibility argument similar to the one used in the proof of Lemma \ref{lem:deg}, for a general $s\in S$ either the set
\[ \Phi_{s}(p_{1}) \ =_{\textup{def}} \ \{ p_{2} \in C_2: (p_{1},p_{2})\in \varphi^{-1}(s)\}\subseteq C_2 \]
$(i)$ does not contain elements of $|L_2|$ for all $p_{1}\in \textup{pr}_1(\varphi^{-1}(s))$, or $(ii)$ is uniquely partitioned into elements of $|L_2|$ for all $p_{1} \in \textup{pr}_1(\varphi^{-1}(s))$. In case $(i)$, since 
\[ g(C_2)\geq \max\{ (a_2-1)(l-1),l \}\]
there is a $1$-form $\omega_2$ on $C_2$ vanishing at all points of $\Phi_s(x)$ except for one. In particular, the $2$-form
\[ \pr_{1}^{\ast} \omega_1 \wedge \pr_{2}^{\ast}\omega_2 \in H^0(C_1\times C_2,\Omega^2) \]
vanishes at all but one point of $\varphi^{-1}(s)$. By tracing to $S$, it follows that $0\neq \Tr_\varphi(\pr_{1}^{\ast} \omega_1 \wedge \pr_{2}^{\ast}\omega_2)\in H^0(S,\Omega^2)$, contradicting the assumption that $p_g(S)=0$.

In case $(ii)$, we can use Lemma \ref{facto} to deduce that the rational map $\varphi$ factors through
\[ \id_{C_1}\times f_{2}  \colon  C_1\times C_2 \longrightarrow C_1\times \mathbb{P}^1.\]

Let us now consider case (2a). For $x \in C_{1} \times C_{2}$, let $D_{1,x}$ be the element of $|L_1|$ which contains $\text{pr}_{1}(x)$. Consider a generic $s\in S$. If for all $x\in \varphi^{-1}(s)$ and all $p_{1},q_{1}\in \Supp(D_{1,x})$ we have
\[ \Phi_s(p_{1})=\Phi_s(q_{1}), \]
then $\varphi^{-1}(s)$ is partitioned into subsets of the form $D\times \{ z \}$ for some $D\in |L_1|$ and some $z \in C_2$. This is depicted in Figure~\ref{figure2a}.

\begin{figure}[htbp!]
\includegraphics[width=6cm]{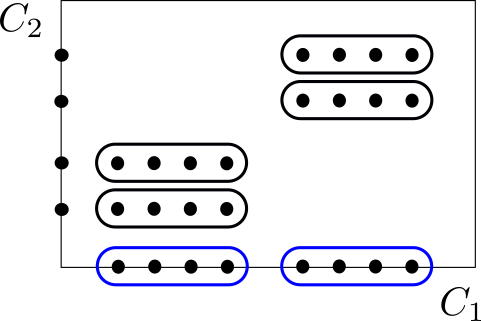}
\caption{Case (2a) when $\Phi_{s}(p_{1}) = \Phi_{s}(q_{1})$ for all $x\in \pr_1(\varphi^{-1}(s))$ and $p_{1}, q_{1} \in D_{1,x}$.}
\label{figure2a}
\end{figure}

\noindent By Lemma \ref{facto}, the map $\varphi$ must factor through
\[ f_1\times\id_{C_2} \colon C_1\times C_2\longrightarrow \mathbb{P}^1\times C_2. \]

Otherwise, pick $x\in \varphi^{-1}(s)$ and $p_{1},q_{1}\in D_{1,x}$ such that $\Phi_s(p_{1})\neq \Phi_s(q_{1})$, and without loss of generality consider
\[ z_{2} \in \Phi_s(p_{1}) \setminus \Phi_s(q_{1}).\]
Since $g(C_2)\geq 2l+1$ and $\Phi_s(p)$ does not contain an element of $|L_2|$, we can find a $1$-form $\omega_2$ on $C_2$ which vanishes everywhere on $\Phi_{s}(p_{1}) \cup \Phi_{s}(q_{1})$
aside from $z$. Then the $2$-form
\[ \pr_{1}^{\ast} \omega_1 \wedge \pr_{2}^{\ast} \omega_2 \in H^0(C_1\times C_2,\Omega^2) \]
vanishes at all but one point of $\varphi^{-1}(s)$, and tracing it to $S$ contradicts the fact that $p_g(S)=0$. 
\end{proof}

This leaves the remaining case (2b). In this setting, we will prove a series of lemmas that describe some strong constraints on the fibers of such a map $\varphi: C_1\times C_2\dashrightarrow S$.

\begin{lemma}\label{lem:T}
Suppose that $\varphi: C_1\times C_2\dashrightarrow S$ falls under case \textup{(2b)}. Then a generic fiber $\varphi^{-1}(s)$ is canonically partitioned into subsets $T_\alpha$ which are indexed by a set $J(s)$ and of the form $(D_{1,x}\times  D_{2,x})\cap \varphi^{-1}(s)$ such that:
\begin{itemize} \setlength\itemsep{0em}
    \item $\abs{T_\alpha}=\abs{T_\beta}$ for all $\alpha,\beta\in J(s)$;
    \item $\pr_i(T_\alpha)\in |L_i|$;
    \item The map $\pr_i: T_\alpha\longrightarrow C_i$ is a covering onto its image. The degree of this covering is independent of $\alpha$.
\end{itemize}
\end{lemma}
\begin{proof}
For any point $x \in \varphi^{-1}(s)$, the $T_{\alpha}$ containing $x$ is given as $(D_{1,x}\times  D_{2,x})\cap \varphi^{-1}(s)\subseteq C_1\times C_2$. The first claim follows from the fact that we would otherwise be able to canonically single out the union of the $T_{\alpha}$ with the largest cardinality, thereby violating irreducibility of $C_1\times C_2$. The second claim follows from the definition of the $T_{\alpha}$ and the fact that we are in case (2b). For the third claim, consider the projection $\pr_{i} \big|_{T_{\alpha}}$ onto its image and collect the union of the largest cardinality fibers in $T_{\alpha}$ for all $\alpha\in J$. By an irreducibility argument, this must be all of $\varphi^{-1}(s)$ and thus $\pr_{i} \big|_{T_{\alpha}}$ is a covering of its image.
\end{proof}

\begin{definition}
A subset $T$ of the Cartesian product of two sets $A$ and $B$ is called \textit{checkered} if there are equal size partitions $A=A_1\sqcup A_2$, $B=B_1\sqcup B_2$ such that
$$T=(A_1\times B_1)\sqcup (A_2\times B_2)\subseteq A\times B.$$
\end{definition}

\begin{lemma}
Suppose that $\varphi$ falls under case \textup{(2b)} and let $T_\alpha$, $\alpha\in J(s)$ be as defined in Lemma \ref{lem:T} for a generic fiber $\varphi^{-1}(s)$. Then either
\begin{itemize}\setlength\itemsep{0em}
    \item $T_\alpha=\textup{pr}_1(T_\alpha)\times \textup{pr}_2(T_\alpha)$ for all $\alpha\in J(s)$, or
    \item $T_\alpha$ is a checkered subset of $\textup{pr}_1(T_\alpha)\times \textup{pr}_2(T_\alpha)$ for all $\alpha\in J(s)$.
\end{itemize}
\end{lemma}
\begin{proof}
Fix $\alpha\in J(s)$. In order to describe which points of $\textup{pr}_1(T_\alpha)\times \textup{pr}_2(T_\alpha)$ are in $T_\alpha$ and which points are not, we will represent $T_\alpha$ as a rectangular box with a lattice of circles. The rows corresponds to points of $\textup{pr}_2(T_\alpha)$ and the columns to points of $\textup{pr}_1(T_\alpha)$. If the circle in the entry corresponding to $(p_1,p_2)\in \textup{pr}_1(T_\alpha)\times \textup{pr}_2(T_\alpha)$ is filled in then $(p_1,p_2)\in T_\alpha$. For instance, the diagram
\begin{equation*} \def\arraystretch{1.5}
\begin{array}{ccccc}
    & p_{1} & q_{1} & r_{1} & s_{1} \\
    p_{2} & \tikzmarknode{1a11}{\bullet} & \circ & \bullet & \circ \\
    q_{2} & \circ & \bullet & \circ & \tikzmarknode{1a25}{\bullet}
\end{array}
    \begin{tikzpicture}[overlay,remember picture]
    \node[fit=(1a11) (1a25),draw=black,fill opacity=0]{};
    \end{tikzpicture}
\end{equation*}
corresponds to $T_\alpha=\{ (p_1,p_2),(q_1,q_2),(r_1,p_2), (s_1,q_2) \}$, and $T_{\alpha}$ is an example of a checkered subset of $\pr_1(T_\alpha)\times \pr_2(T_\alpha)$. We will mostly avoid labelling the columns and rows and allow ourselves to permutes rows and columns as we please.

Consider any two points $p_{i}, q_{i} \in \pr_{i}(T_{\alpha})$ for $i = 1, 2$. Suppose that the set $T_\alpha\cap (\{p_1,q_1\}\times \{p_2,q_2\})$ has the following configuration:
\begin{equation}\label{config1} \def\arraystretch{1}
\begin{array}{cc}
 \tikzmarknode{1a22}{\bullet} & \circ \\
\circ & \tikzmarknode{1a33}{\circ}
\end{array}
    \begin{tikzpicture}[overlay,remember picture]
    \node[fit=(1a22) (1a33),draw=black,fill opacity=0]{};
    \end{tikzpicture}
\end{equation}
For $i=1,2$, by Lemma~\ref{lem:pq} we can find a form $\omega_{i} \in H^{0}(C_{i}, \Omega^{1}_{C_{i}})$ which vanishes everywhere on $\pr_{i}(\varphi^{-1}(s))$ except at $p_{i}$ and $q_{i}$. The 2-form $\pr_{1}^{\ast}\omega_{1} \wedge \pr_{2}^{\ast}\omega_{2}$ then vanishes everywhere on $\varphi^{-1}(s)$ except at the point which corresponds to the filled in circle in diagram \eqref{config1}, thereby contradicting the assumption $p_g(S)=0$. Similarly, suppose that the set $T_\alpha\cap (\{p_1,q_1\}\times \{p_2,q_2\})$ has the following configuration:
\begin{equation}\label{config2} \def\arraystretch{1}
\begin{array}{cc}
\tikzmarknode{1a22}{\circ} & \bullet \\
 \bullet & \tikzmarknode{1a33}{\bullet}
\end{array}
\begin{tikzpicture}[overlay,remember picture]
    \node[fit=(1a22) (1a33),draw=black,fill opacity=0]{};
\end{tikzpicture}
\end{equation}
If $(p_1,p_2)$ is the point corresponding to the empty circle in \eqref{config2} the first configuration in \eqref{config1} must appear in $T_\beta$ for some $\beta\in J(\varphi(p_1,p_2))$. Again, we are able to find a form which vanishes everywhere on $\varphi^{-1}(\varphi(p_1,p_2))$ aside from a point, which contradicts the assumption $p_g(S)=0$. Indeed, given $s\in S$ generic, any $\alpha\in J(s)$, and any choice of $p_1,q_1\in \pr_1T_\alpha$ and $p_2,q_2\in \pr_2 T_\alpha$ giving rise to configuration \eqref{config2}, the point $\varphi(p_1,p_2)\in S$ is also generic.

Thus, we may assume that for any choice of $p_1,q_1\in \pr_1T_\alpha$ and $p_2,q_2\in \pr_2 T_\alpha$ the set $$T_\alpha\cap (\{p_1,q_1\}\times \{p_2,q_2\})$$
consists of $0,2,$ or $4$ points. After permuting the rows and columns of the diagram of $T_\alpha$ we can assume that both the first row and first column consist of a series of $\bullet$ followed only by $\circ$ as pictured in the first diagram of Figure~\ref{Figure:T_alpha}. The fact that the configurations \eqref{config1} and \eqref{config2} do not arise in this diagram places constraints on what the rest of this diagram can look like.

It is useful to distinguish two cases. In the first case the entire first row consists only of $\bullet$ or only of $\circ$. It is easy to see that the fact that configurations of the form \eqref{config1} and \eqref{config2} do not arise forces the entire diagram to consist of a series of rows filled with $\bullet$ followed by a series of rows filled with $\circ$. Since by assumption $T_\alpha\neq \emptyset$ and the projection of $T_\alpha$ to the second factor must surject onto $\text{pr}_2(T_\alpha)$, there can be no rows consisting entirely of $\circ$ and thus $T_\alpha=\textup{pr}_1(T_\alpha)\times \textup{pr}_2(T_\alpha)$.

The second case where the first row consists of a positive number of $\bullet$ followed by a positive number of $\circ$ is more interesting:
\begin{enumerate}
    \item Consider the second row of the diagram from the left. If the second entry of this row from the top were $\circ$ then the $2\times 2$ square made up of the first two entries of the first two columns would be 
    \[ \def\arraystretch{1}
\begin{array}{cc}
\tikzmarknode{1a22}{\bullet} & \bullet \\
 \bullet & \tikzmarknode{1a33}{\circ} \\
\end{array}
\begin{tikzpicture}[overlay,remember picture]
    \node[fit=(1a22) (1a33),draw=black,fill opacity=0]{};
\end{tikzpicture}
    \]
This would contradict the parity statement about sets of the form $T_\alpha\cap (\{p_1,q_1\}\times \{p_2,q_2\})$. The second entry of the second row must therefore be $\bullet$. Moving down the second column, the same reasoning shows that every row containing a $\bullet$ in the first column contains a $\bullet$ in the second column. If the first row from the top containing a $\circ$ in the first column also contained a $\bullet$ in the second column we would get the following forbidden configuration:
 \[ \def\arraystretch{1}
\begin{array}{cc}
\tikzmarknode{1a22}{\bullet} & \bullet \\
 \circ & \tikzmarknode{1a33}{\bullet} \\
\end{array}
\begin{tikzpicture}[overlay,remember picture]
    \node[fit=(1a22) (1a33),draw=black,fill opacity=0]{};
\end{tikzpicture} \]
Going down the second column further, the same reasoning then shows that every row containing $\circ$ in the first column also contains $\circ$ in the second.

\item We can repeat the argument used in (1) to show that every column starting with a $\bullet$ consists of a series of $\bullet$ followed by a series of $\circ$. Moreover, the number of $\bullet$ is constant. The third diagram in Figure~\ref{Figure:T_alpha} illustrates what we know about the diagram at this point of the argument.

\item Now consider the first column from the left which starts with a $\circ$. This is also the first column which we have yet to describe. If the first entry of the second row were a $\bullet$ and the second entry of this column were a $\bullet$ we would get a forbidden configuration of the form
 \[ \def\arraystretch{1}
\begin{array}{cc}
\tikzmarknode{1a22}{\bullet} & \circ \\
 \bullet & \tikzmarknode{1a33}{\bullet} \\
\end{array}
\begin{tikzpicture}[overlay,remember picture]
    \node[fit=(1a22) (1a33),draw=black,fill opacity=0]{};
\end{tikzpicture} \]
Repeating this argument, we can check that every row which starts with a $\bullet$ contains a $\circ$ in this column and every row that starts with a $\circ$ contains a $\bullet$ in that column. Moving to the next column and repeating the same argument shows that $T_\alpha$ must be as pictured in the last diagram of Figure \ref{Figure:T_alpha}. It follows that $T_\alpha$ is a checkered subset of $\textup{pr}_1(T_\alpha)\times \textup{pr}_2(T_\alpha)$.
\end{enumerate}

\begin{figure}[htbp!]
\[
\begin{array}{cccccc}
    \tikzmarknode{1a11}{\bullet} & \cdots & \bullet & \circ & \cdots & \circ \\
    \vdots & & & & & \\
    \bullet & & & & & \\
    \circ & & & & & \\
    \vdots & & & & & \\
    \circ & & & & & \tikzmarknode{1a66}{\phantom\bullet}
\end{array}
\underset{(1)}{\longleadsto{1}}
\begin{array}{ccccccc}
    \tikzmarknode{2a11}{\bullet} & \bullet & \cdots & \bullet & \circ & \cdots & \circ \\
    \vdots & \vdots & & & & & \\
    \bullet & \bullet & & & & & \\
    \circ & \circ & & & & & \\
    \vdots & \vdots & & & & & \\
    \circ & \circ & & & & & \tikzmarknode{2a77}{\phantom\bullet}
    \begin{tikzpicture}[overlay,remember picture]
    \node[fit=(1a11) (1a66),draw=black,fill opacity=0]{};
    \node[fit=(2a11) (2a77),draw=black,fill opacity=0]{};
    \end{tikzpicture}
\end{array}
\underset{(2)}{\longleadsto{1}}
\]

\[
\underset{(2)}{\longleadsto{1}}
\begin{array}{cccccc}
    \tikzmarknode{3a11}{\bullet} & \cdots & \bullet & \circ & \cdots & \circ \\
    \vdots & & \vdots & & & \\
    \bullet & \cdots & \bullet & & & \\
    \circ & \cdots & \circ & & & \\
    \vdots & & \vdots & & & \\
    \circ & \cdots & \circ & & & \tikzmarknode{3a66}{\phantom\bullet}
\end{array}
\underset{(3)}{\longleadsto{1}}
\begin{array}{cccccc}
    \tikzmarknode{4a11}{\bullet} & \cdots & \bullet & \circ & \cdots & \circ \\
    \vdots & & \vdots & \vdots & & \vdots \\
    \bullet & \cdots & \bullet & \circ & & \circ \\
    \circ & \cdots & \circ & \bullet & \cdots & \bullet \\
    \vdots & & \vdots & \vdots & & \vdots \\
    \circ & \cdots & \circ & \bullet & \cdots & \tikzmarknode{4a66}{\bullet}
    \begin{tikzpicture}[overlay,remember picture]
    \node[fit=(3a11) (3a66),draw=black,fill opacity=0]{};
    \node[fit=(4a11) (4a66),draw=black,fill opacity=0]{};
    \end{tikzpicture}
\end{array}
\]
\caption{A picture of the diagram for $T_{\alpha} \subseteq \pr_{1}(T_{\alpha}) \times \pr_{2}(T_{\alpha})$.}
\label{Figure:T_alpha}
\end{figure}
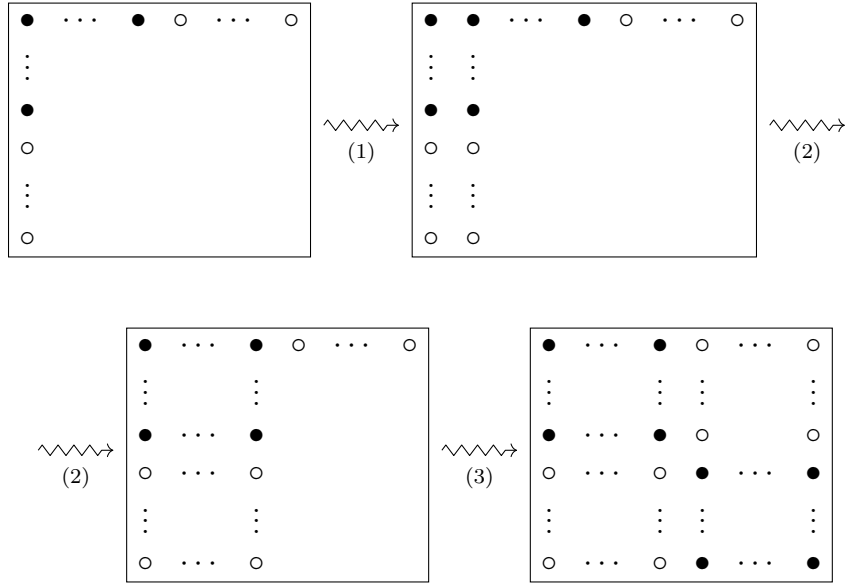

\end{proof}

We are finally ready to resolve the last case of Theorem \ref{thm:largeg}. 
\begin{proposition}
Theorem \ref{thm:largeg} holds if $\varphi: C_1\times C_2\dashrightarrow S$ falls under case (2b) of Lemma \ref{lem:cases}.
\end{proposition}

\begin{proof}
Consider a general $s\in S$ and $\alpha\in J(s)$. We may assume that $T_\alpha$ is checkered as a subset of $\pr_1(T_\alpha)\times\pr_2(T_\alpha)$. This allows us to write $a_1=2a_1'$ for some positive integer $a_1'$. Indeed, if $T_\alpha=\textup{pr}_1(T_\alpha)\times \textup{pr}_2(T_\alpha)$ for all $\alpha\in J$ then $\varphi$ factors through $f_1\times f_2: C_1\times C_2\longrightarrow \mathbb{P}^1\times \mathbb{P}^1$ by Lemma \ref{facto}. The idea is that this checkered structure on $T_\alpha$ may be used to find an algebraically defined partition of the elements of $|L_1|$ into two subsets of equal sizes. Such a partition gives rise to a factorization of the gonal map $f_1$ associated to the pencil $|L_1|$.

Let $U\subseteq C_1\times C_2$ be an open on which $\varphi$ is \'etale. Fix a point $c\in \text{pr}_2(U)$ and consider the subvariety
\begin{align*}
Z' &=_{\textup{def}} \overline{\left\{ \let\scriptstyle\textstyle \substack{ \left[\pr_1\left((D_{1,(x_1,c)}\times D_{2,(x_1,c)})\cap F_{(x_1,c)}\right) \right] + \\ \left[ D_{1,(x_1,c)}\setminus \pr_1\left((D_{1,(x_1,c)}\times D_{2,(x_1,c)})\cap F_{(x_1,c)}\right) \right]} \ \middle| \ x_1\in \text{pr}_1(U)\right\}} \\
&\subseteq \ \Sym^2(\Sym^{a_1'}(C_1)).
\end{align*}
The following lemma then implies that the gonal map $C_1\longrightarrow \mathbb{P}^1$ must factor, which provides the desired contradiction.
\end{proof}

\begin{lemma}
Consider a curve $C$ with a morphism $f \colon C \longrightarrow \mathbb{P}^1$ of degree $k$. Write $Z\subseteq \Sym^k(C)$ for the corresponding linear system. Suppose that $k=d k'$ for some integers $d$ and $k'$ and that there exists a subvariety $Z'\subseteq \Sym^d(\Sym^{k'}(C))$ along with a birational map $\nu \colon Z' \rightarrow Z$ which fit into the commutative diagram
\begin{center}
\begin{tikzcd}
Z' \arrow[d, swap, "\nu"] \arrow[r, hook] & \Sym^{d}(\Sym^{k'}(C)) \arrow[d] \\
Z \arrow[r, hook] & \Sym^{k}(C).
\end{tikzcd}
\end{center}
Then the map $f$ factors through a map $C \dashrightarrow C'$ of degree $k'$, for some curve $C'$.
\end{lemma}

\begin{proof}
Consider the incidence correspondence
\[
I=_{\textup{def}}\left \{(c,\mathbf{x}=_{\textup{def}}x_1+\ldots+x_{k'}) \middle| \ \let\scriptstyle\textstyle \substack{c=x_1, \ \exists \  \mathbf{y}_1, \ldots, \mathbf{y}_{d-1}\in \Sym^{k'}(C), \\ {\text{and} \quad \mathbf{x}+\mathbf{y}_1+\ldots+\mathbf{y}_{d-1}} \in Z'} \right\} \subseteq C\times \text{Sym}^{k'}(C).
\]
Note that $I$ maps birationally onto $C$ under the first projection and generically finitely onto its image $C'$ under the second projection. The projection $C\longrightarrow C'$ has degree $k'$ and the fibers of $C\longrightarrow \mathbb{P}^1$ can be partitioned by the fibers of this map. By Lemma~\ref{facto}, we have a factorization as required.
\end{proof}

This completes the proof of Theorem~\ref{thm:largeg}.

\section{Further questions}

The results in this paper are all obtained under the assumption that $p_{g}(S) = 0$ by tracing $2$-forms from $C_{1} \times C_{2}$ to $S$. In this section, we will sketch our original approach which proceeded along the lines of \cite{Martin19} but was eventually absorbed by Theorem \ref{thm:largeg} and Theorem \ref{thm:hyp}.

Given a dominant rational map $\varphi: C_1\times C_2\dashrightarrow S$ of degree $k$ to a surface with $p_g(S)=q(S)=0$ and an open subset $U\subseteq C_1\times C_2$ such that $\varphi|_{U}$ is \'etale on its image, one may consider the cycles
\[Z_l=\overline{\{(x_1,\ldots, x_l)\in U^l: \varphi(x_i)=\varphi(x_j), x_i\neq x_j, \forall i,j, i\neq j \}}\subseteq (C_1\times C_2)^l.\]
The fact that the cycle class of the diagonal $\Delta_S$ belongs to the subset
\[ [S\times \text{pt}]+[\text{pt}\times S]+\textup{Sym}^2(NS(S))\subseteq H^4(S\times S,\mathbb{Z}) \]
imposes constraints on $[Z_2]$. Moreover, since $Z_k$ is in a fiber of the Abel-Jacobi map
\[ \alpha: (C_1\times C_2)^k\longrightarrow \textup{Alb}(C_1\times C_2), \]
the cohomology class $\alpha_*([Z_k])$ vanishes. It turns out that for $k=3$ these cohomological contraints together imply that
$\varphi^{-1}(\varphi|_U((C_1\times{\{x\}})\cap U))$ is a union of fibers of the second projection $C_1\times C_2\longrightarrow C_2$ for generic $x\in C_2$, and similarly with $C_1$ and $C_2$ interchanged. This is enough to show that the degree of irrationality of two hyperelliptic curves is $4$. However, for higher degree rational maps we were unable to salvage conclusive constraints from the vanishing of $\alpha_{*}([Z_k])$.

We leave the reader with some natural questions and problems.

\begin{problem}
Under similar hypothesis, extend Theorem~\ref{thm:largeg} to a product of $n$ curves.
\end{problem}

\begin{question}\label{questions}
Let $C_1, \ldots, C_{n}$ be very general curves:
\begin{enumerate}
    \item Is the degree of irrationality of $C_{1} \times \cdots\times C_n$ equal to $\gon(C_{1}) \cdots \gon(C_n)$?
    \item Is every rational map $C_1\times\cdots\times C_n\dashrightarrow \mathbb{P}^n$ of degree $\gon(C_{1}) \cdots \gon(C_{n})$ birationally equivalent to a product of rational maps $C_i\dashrightarrow \mathbb{P}^1$?
    \item Let $E$ be a very general elliptic curve and $C$ be a very general curve of genus $g \geq 1$. Does $E\times C$ admit a dominant rational map of degree $3$ to $\mathbb{P}^2$? More generally, does $E\times C$ admit a dominant rational map of degree $3$ to a surface with $p_g=0$?
\end{enumerate}
\end{question}

In the proofs of Theorems~\ref{thm:hyp} and \ref{thm:largeg} we only trace decomposable $2$-forms on $S$, namely forms $\text{pr}_1^*\omega_1\wedge \text{pr}_2^*\omega_2$, where $\omega_i$ is a $1$-form on $C_i$. Such $2$-forms make up a $(g(C_1)+g(C_2))$-parameter family whereas
\[ \dim H^0(C_1\times C_2,\Omega^2)=g(C_1)\cdot g(C_2). \] We expect that a lot more information can be captured by tracing an arbitrary $2$-form on $C_1\times C_2$.

For instance if $C_1$ and $C_2$ are very general curves of genus $2$ and $3$, we can show along the lines of Theorem~\ref{thm:hyp} that $\textup{irr}(C_1\times C_2)\geq 5$. On the other hand, we expect $\irr(C_{1} \times C_{2}) = 6$, so it suffices to rule out the possibility of a degree $5$ dominant rational map $\varphi: C_1\times C_2\dashrightarrow \mathbb{P}^2$. Consider the composition
\[ \Psi: C_1\times C_2 \ \xrightarrow{\ \psi_1\times \psi_2 \ } \  \mathbb{P}^1\times \mathbb{P}^2 \ \xrightarrow{\ \sigma \ }  \ \mathbb{P}^5, \]
where $\psi_i$ is the canonical map of $C_i$ and $\sigma$ is the Segre embedding. Given $t\in \mathbb{P}^2$ generic, we can use the fact that the points in $\Psi(\varphi^{-1}(t))$ fail to impose independent conditions on $2$-forms on $C_1\times C_2$ to show that they span a codimension $2$ linear subspace of $\mathbb{P}^5$. This linear subspace meets $\sigma(\mathbb{P}^1\times \mathbb{P}^2)$ in a curve whose projection to $\mathbb{P}^2$ is a conic. This conic $Q_t$ meets the quartic $\psi_2(C_2)\subseteq \mathbb{P}^2$ in $8$ points, 5 of which are the points of $\psi_2(\textup{pr}_2(\varphi^{-1}(t)))$. One can check that the three residual points $a$, $b$, and $c$ do not vary with $t$, so $\psi_2(\textup{pr}_2(\varphi^{-1}(t)))$ is contained in the linear system $|2K_{C_2}-a-b-c|$. Although this does not give an immediate contradiction, we do get strong constraints on $\varphi$ which cannot be obtained by only tracing decomposable $2$-forms.


\bibliographystyle{siam} 
\bibliography{Biblio}

\end{document}